\documentclass[12pt,a4paper]{article}
\usepackage{amssymb,amsfonts,amsmath,amsthm,graphicx}
\usepackage[latin1]{inputenc}

\theoremstyle{plain}
\newtheorem{theorem}{Theorem}[section]
\newtheorem{lemma}[theorem]{Lemma}
\newtheorem{proposition}[theorem]{Proposition}
\newtheorem{remark}[theorem]{Remark}
\newtheorem{corollary}[theorem]{Corollary}
\newtheorem{definition}[theorem]{Definition}
\newtheorem{example}[theorem]{Example}

\begin{document}
\today

\pagestyle{plain}

\pagenumbering{arabic}

\begin{center}\large{\bf Description of Partial Actions}\footnote{ The second named author was supported by Fapesp, projeto tem\'atico 2014/09310-5
}\end{center}

\begin{center}
{\bf {\rm Wagner Cortes$^1$, Eduardo N. Marcos$^2$}}\end{center}

\begin{center}{ \footnotesize 1 Instituto de Matem\' atica\\
Universidade Federal do Rio Grande do Sul\\
91509-900, Porto Alegre, RS, Brazil\\
2, Departamento de Matematica\\
IME-USP, Caixa Postal 66281,\\
05315-970, S\~ao Paulo-SP, Brazil\\
E-mail: {\it wocortes@gmail.com}, {\it enmarcos@ime.usp.br}}\end{center}

\begin{abstract}
In this paper, we study partial actions of groups  on $R$-algebras, where $R$ is a commutative ring.
We describe the partial actions of groups  on the  indecomposable  algebras with enveloping actions. Then we work on algebras that can be decomposed  as product of  indecomposable algebras and we  give a description of the partial actions of groups on these algebras  in terms of global actions.
\end{abstract}

\section{Introduction}

Partial actions of groups were first proposed by R. Exel and other authors in the context of $C^{*}$-algebras, see for example, [13], and afterwards they appeared in a pure algebraic  setting, see  \cite{DE}. A different way of looking at partial actions, appeared earlier in the work of Green and Marcos \cite{GM}. 

Partial actions of groups became an important tool to characterize algebras as partial crossed products. In particular,  several aspects of Galois theory can be generalized to
partial group actions, see \cite{DFP} (at least under the additional
assumption that the associated ideals are generated by central
idempotents). Soon after the initial definition,  the theory of partial actions was extended to the Hopf  algebraic setting, see \cite{CJ}  and \cite{ME}.

An interesting class of partial actions is the one which can  be obtained by a restriction of a global action. In this class  we have the possibility to transfer properties of the global action to the  partial action.  The problem of existence of globalization was firstly studied in
\cite{abadie} and many other  results concerning globalization were obtained
later on (see, for instance, \cite{CF},
\cite{DE}, \cite{DRS}, \cite{Ferrero2}).  
This was also studied in the setting of categories by the authors in \cite{CFM}.

In this paper $R$ will always denote a  commutative  ring with unity and all  algebras are $R$-algebras. Observe that for us in this paper all  $R$-algebras  does not necessarily   have a unity unless otherwise stated. Moreover,  for an algebra $\Lambda$ there is a $R$-module structure compatible with the multiplications, that is $r(m.n)= (rm).n= (m).(rn)$ for all $r\in R$, $m$ and $n \in \Lambda$.

The idea underlying this article is to consider partial actions of groups on algebras, which have a globalization.

The contents of the article are as follows.

In the Section 2, we present some preliminary results and definitions which will be used in this paper.

 In the Section 3,  we present our principal results. We  begin with the  notion of extension by zero  of a partial action and we explicitly give the description of  the enveloping actions of partial actions that are extension by zero of the global actions. In a certain way, which is clear in the paper, we give a characterization of all the partial actions with enveloping actions  in indecomposable algebras and product of them.  Moreover, we give an explicit characterization of  enveloping actions of partial actions on certain algebras that are product of indecomposable algebras.

\section{Preliminaries}

In this section we present some definitions and results that will be used in the rest of the paper. 

In \cite{DE} the authors introduced  the following definition. 

\begin{definition} \label{def1} Let $G$ be a group and $\Lambda$ an
$R$-algebra. A {\it partial
 action} $\alpha$ of $G$ on $\Lambda$ is a collection of ideals
 $S_g$ of $\Lambda$, $g\in G$,  and isomorphisms of (non-necessarily unital)
 $R$-algebras $\alpha_g:S_{g^{-1}}\to S_g$ such
 that:

 (i) $S_e=\Lambda$ and $\alpha_e$ is the identity map of $\Lambda$, where $e$ is the identity of $G$;

 (ii) $S_{(gh)^{-1}}\supseteq \alpha_h^{-1}(S_h\cap S_{g^{-1}})$;

(iii) $\alpha_{g}\circ \alpha_{h}(x)=\alpha_{gh}(x)$, for every $x\in
\alpha^{-1}_{h}(S_{h}\cap S_{g^{-1}})$.

\end{definition}

It is convenient to point out that  the property (ii) of the definition above  easily implies that $\alpha_{g}(S_{g^{-1}}\cap S_{h})=S_{g}\cap S_{gh}$,
 for all $g,h\in G$. Also $\alpha_{g^{-1}}=\alpha^{-1}_{g}$, for every $g\in
 G$.

 Next, we see a natural example of a  partial action   where we restrict a  global action  to  an ideal generated by a central idempotent.
 
\begin{example}  Let $\beta$  be a  global action of a group $G$
on a (non-necessarily unital) ring $L
$ and $\Lambda$ an ideal of $L$
generated by a central idempotent $1_{\Lambda}$. We can restrict $\beta$
to $\Lambda$ as follows: putting  $S_g = \Lambda \cap \beta_{g}(\Lambda) = \Lambda
\beta_{g}(\Lambda)$, $g\in G$, each $S_{g}$ has an identity element
$1_{\Lambda}\beta_{g}(1_{\Lambda})$. Then defining
$\alpha_{g}=\beta_{g}|_{S_{g^{-1}}}$, for all $g\in G$, the items ($i$),
($ii$) and ($iii$) of  Definition \ref{def1} are satisfied. So,  $\alpha=\{\alpha_g:S_{g^{-1}}\rightarrow S_g:g\in G\}$ is a partial action of $G$ on $\Lambda$. \end{example}

The  following definition appears in (\cite{DE}, pg. 9).

\begin{definition}\label{def2}
\emph{A  global action $\big(L, \{\beta_g\}_{g\in G} \big)$
of a group $G$ on an associative (non-necessarily unital) ring $L$ is said to be an
\emph{enveloping action} (or a \emph{globalization}) for a partial action $\alpha$
of $G$ on a ring $\Lambda$ if there exists a ring monomorphism $\varphi:\Lambda\rightarrow L$ such that the following
properties hold. In this case the algebra $L$ is called the \emph{enveloping algebra}.
}
\begin{itemize}
\item [$(i)$] $\varphi(\Lambda)$  \emph{is an  ideal of} $L$;
\item [$(ii)$] $L = \sum_{g\in G}\beta_g(\varphi(\Lambda))$;
\item [$(iii)$] $\varphi(S_g) = \varphi(\Lambda)\cap \beta_g(\varphi(\Lambda))$, for all $g$ in $G$;
\item [$(iv)$] $\varphi\circ \alpha_g(a) = \beta_g\circ \varphi(a)$, \emph{for all} $a\in S_{g^{-1}}$, for all $g$ in $G$.

\end{itemize}
\end{definition}
\vskip3mm

%Given a partial action $\alpha$ of a group $G$ on $\Lambda$, an
%enveloping action is an algebra $T$ together with a global action
%$\beta=\{\sigma_g \mid g\in G\}$ of $G$ on $T$, where for every $g\in G$,   $\sigma_g$
%is an automorphism of $T$, such that the partial action is given
%by restriction of the global action, see (\cite{DE}, Definition 4.2).

The following theorem appears in  (\cite{DE}, Theorem 4.5) and it is central on the study of the partial actions of groups  on algebras. 

\begin{theorem} \label{enveloping action}
A partial action $\alpha$ of a group $G$ on a unital $R$-algebra $\Lambda$ admits an enveloping action $(L,\beta)$ if and only if each ideal $S_g$, $g\in G$, is generated by a central idempotent. Moreover,  if the enveloping action exists it is unique up to equivalence, i.e, if there exists an $R$-algebra  $L'$ with a global action $\beta'$ of $G$ such that the properties of Definition \ref{def2} are satisfied,  then $L'\simeq L$. 
\end{theorem}

When $(\Lambda,\alpha)$ has an enveloping action $(L, \beta)$ we may consider
that $\Lambda$ is an ideal of $L$ and the following properties hold:

(i) The subalgebra of $L$ generated by $\bigcup_{g\in G}
\beta_g(\Lambda)$ coincides with $L$ and we have $L=\sum_{g\in
G} \beta_g(\Lambda)$;

(ii) $S_g=\Lambda\cap \beta_g(\Lambda)$, for every $g\in G$;

(iii) $\alpha_g(x)=\beta_g(x)$, for every $g\in G$ and $x\in
S_g^{-1}$.

\section{Description of the partial actions}
 In this section, we present our main results and we  start with the following definition  which plays a role on this work.

\begin{definition} \label{trivialextension}
{\bf Extension by zero }

a) Let $H$ be a subgroup of a group $G$ and $\alpha$ a  partial action of $H$ on an algebra $\Lambda$. The  extension by zero to $G$ of $(\Lambda,\alpha, H)$, is the partial action $\gamma$ of $G$ defined as follows:
\begin{itemize}
\item[(i)]
 $\left\{ \begin{array}{c}
  I^{\gamma}_{g}=0 \,\, if \,\, g\notin H\\ 
I^{\gamma}_{g}=I_g^{\alpha} \,\, if \,\, g\in H
\end{array}\right.$ 

\item[(ii)]
 $\left\{\begin{array}{c}
 \gamma_g=0 \,\, if \,\, g\notin H\\ 
\gamma_g=\alpha_g\,\, if \,\, g\in H
\end{array}\right.$  
\end{itemize} 
\end{definition}

\begin{remark}

It is convenient to point out that  in the case $H$ acts globally  on $\Lambda$  we say that the partial action defined above is an extension by zero of the $H$-action. We also say that a partial
action $\alpha$ of a group $G$ is an extension by zero of an action if it is a extension by zero of an $H$-action for some  subgroup $H$ of $G$. Observe that a global action is a particular case where $H=G$.
\end{remark}

It is convenient to point  out that the definition above shows an advantage of partial  actions, i.e.,  partial  actions of subgroups always can be extended to the all group, which is not the case for  global actions. 

We are ready to show  our first lemma.

\begin{lemma}\label{globalrestriction}
 Let   $\alpha$ be a partial  action of a group $G$ on a unital  $R$-algebra $\Lambda$. Then the set $H=\{h\in G| I_h=\Lambda\}$ is a subgroup of $G$ and the restriction of  the partial action to $H$ is a global action.
 \end{lemma}
 
\begin{proof} 
Let $h,t\in H$. Then , $I_t=I_h=\Lambda$ and we have that $\alpha_h(I_{h^{-1}}\cap I_t)=I_h\cap I_{ht}=\Lambda$. Thus, $I_{ht}=\Lambda$.  Note that $I_{h^{-1}}=\Lambda$. So,   $H$ is a subgroup of $G$. 

The second part of the statement is clear.
\end{proof}

We recall that an algebra with unity is indecomposable if and only  if it has exactly two central idempotents 0 and 1. Moreover, an algebra without identity is connected if the unique central idempotent is 0.

In the next proposition  it is explained  why we introduced the notion of extension by zero.

\begin{proposition} \label{indec}
Let $\alpha$ be a partial action of the group $G$ on a unital indecomposable algebra $\Lambda$.  Then $\alpha$ has an enveloping action  if and only if there exists a subgroup $H$ of $G$ and an action  of $H$ on $\Lambda$ such that $\alpha$ is the extension by zero of this action.\end{proposition}

 %$\Lambda$ is an indecomposable algebra with one, then all the partial actions of groups  with enveloping action are extension by zero of the global actions.\end{proposition}

\begin{proof} 
Let $\alpha=\{ \alpha_g:S_{g^{-1}}\rightarrow S_g| g\in G\}$ be a partial action of a group $G$ on an indecomposable algebra $\Lambda$ and suppose that $\alpha$ has enveloping action. According to   Theorem \ref{enveloping action} we have  that the  partial action $\alpha$ has a globalization if and only if each ideal $S_g$, $g\in G$, is generated by a central idempotent. Since $\Lambda$ is indecomposable then the unique central idempotents are $0$ and 1 and it follows that  each $S_g$, $g\in G$, is either the zero ideal or the algebra $\Lambda$. So,  the result follows  from  Lemma 
\ref{globalrestriction} .

Conversely, by assumption we have that all the ideals $S_g$ are either the zero ideal or the algebra $\Lambda$. Thus, by Theorem \ref{enveloping action} we have the result.  
 
\end{proof}

The next remark will be useful in our first main result.

\begin{remark}
 Let $H$ be a subgroup of $G$ and $T=\{g_i\}_{i\in I}$ be a  left transversal set of the congruence defined by   $H$, i.e., $G=\dot{\cup}_{i\in I}g_iH$ and we  additionally require  that the element that represents $H$ is the identity $e$. The functions 
  $j:G\times T \rightarrow T$ and $h:G\times T \rightarrow H$ are defined by the equality $gg_i=j(g,g_i)h(g,g_i)$.  Moreover, note that for all $t,g \in G$  we get $tg_i=j(t,g_i)h(t,g_i)$ and $gtg_i=j(gt,g_i)h(gt,g_i)$.
  
  From now on in the paper given any subgroup $H$ of a group $G$ we fix a left transversal $T$ so that the maps $j$ and $h$ above 
  are well defined.
 \end{remark}
 
 % Given a subgroup $H$  of a group $G$ and an $H$-action on an algebra $\Lambda$,
  
   In  the next  theorem  we  completely describe   the enveloping action of  the extensions by zero of global actions of  the    algebra   $\Lambda$. 

\begin{theorem} \label{idemc2}
Let $\Lambda$ be an unital $R$-algebra, $H$ a subgroup of $G$ that acts globally  on $\Lambda$ by the action $h.a$, where $h\in H$ and $a\in \Lambda$,  and $\alpha$  the extension by zero of the  global action of $H$   on $\Lambda$ to $G$  and $\{g_i\}$ a transversal of $G/H$. Then we can describe the enveloping action $(\Gamma, \beta)$ of $(\Lambda, \alpha)$ as follows:

a) $\Gamma=\amalg_{i\in I} \Lambda_{g_{i}}$ with $\Lambda_{g_{i}}=\Lambda;$

b)  For each $g\in G$, $\beta_g(\lambda_{g_i})=(h(g,g_i).(\lambda_{j(g,g_i)})$, where each $\lambda\in \Gamma$ can be seen as a (finitely supported) function $\lambda:T\rightarrow \Gamma$. 

%$g\in G$, $\beta_g(\lambda_{g_i})=(h(g,j(g^{-1},g_i)).\lambda(j(g^{-1},g_i))$. 

%$g\in G$, $\beta_g(\lambda_{g_i})=(h(g,g_i)(\lambda_{g_i}))_{j(g,g_i)}$
\end{theorem}
                              
\begin{proof} 
                              
  Since $eg_i=g_ie$, then we have that $j(e,g_i)=g_i$ and $h(e,g_i)=e$, where $e$ is the identity  element of $G$. Thus, $\beta_e(\lambda_{g_i})=(e.\lambda(g_i))=\lambda_{g_i}$ and we obtain that  $\beta_e=id_{\Gamma}$. We  claim that $\beta_g\beta_t=\beta_{gt}$, for all $g,t\in G$. In fact, note that   $\beta_g(\beta_t(\lambda_{g_i}))=\beta_g(h(t,g_i).\lambda(j(t,g_i)))=(h(g,j(t,g_i))(h(t,g_i).\lambda(j(g,j(t,g_i)))$ and $\beta_{gt}(\lambda_{g_i})=(h(gt,g_i).\lambda(j(gt,g_i)))$.  From the  Remark 3.5  we have that 
  \begin{center} 
  $tg_i=j(t,g_i)h(t,g_i)$ (1) 
   \end{center} 
 and 
 \begin{center}
  $gtg_i=j(gt,g_i)h(gt,g_i)$. (2) 
   \end{center}  
     From (1)  and (2) we have that \begin{center} $gtg_i= g(t,g_i)h(t,g_i)= g(j(t, g_i)h(t,g_i))= j(g,j(t,g_i))h(g,j(t,g_i))h(t,g_i)$.\end{center} Consequently, 
  $j(gt,g_i)=j(g,j(t,g_i))$ and  $h(gt,g_i)=h(g,j(t,g_i))h(t,g_i)$. Hence,   $\beta_g\beta_t=\beta_{gt}$.  One can see  also that  $\beta_g$, for each $g\in G$, is an isomorphim of algebras.  We  claim that  $(\Gamma, \beta)$ is an enveloping action of $(\Lambda, \alpha)$, i.e., we need to show the  four items of Definition 2.3.  In fact, note that  by the construction of  $\Gamma$ we have that $\Lambda$ is an ideal of $\Gamma$ and $\Gamma=\sum_{g\in G} \beta_g(\Lambda)$.   
  
  The next step is to show that  $\beta_g(a)=\alpha_g(a)$ for each $a\in S_{g^{-1}}$.   To do this, let $g\in G$. If $g\notin H$,  then  for each $a\in \beta_g(\Lambda)\cap \Lambda$ we have that   $a=\beta_g(\lambda)=(h(g,e).\lambda(j(g,e))=\lambda_g$. Since the sum $\amalg_{i\in I} \Lambda_{g_{i}}$ is  direct, then $a=\lambda_g=0$. Hence, $\beta_g(\Lambda)\cap \Lambda =0$ and it follows that  $\alpha_g(a)=\beta_g(a)$ for each $a\in S_{g^{-1}}$. Next, suppose that $g\in H$. Then $S_{g^{-1}}=\Lambda$ and we have for each $\lambda \in \Lambda$  that $\beta_g(\lambda)=\beta_g(\lambda(e))=h(g,e).\lambda(j(g,e)=g.\lambda(e)=g.\lambda=\alpha_g(\lambda)$ 
  
    Finally,  let $g\in G$ and we claim that $\beta_g(\Lambda)\cap \Lambda =S_g$. In fact, we have the following two cases:
  
  Case 1: Suppose that $g\in H$. In this case, $S_g=\Lambda$ and we have that $\beta_g(\Lambda)\cap \Lambda \subseteq \Lambda$.  Now, for each  $a\in \Lambda$, we have that   $a=\alpha_g(a')=\beta_g(a')\in \beta_g(\Lambda)\cap \Lambda$. 
  
  Case 2: Suppose that $g\notin H$. In this case, $S_g=0$,  and we obtain as before that $\beta_g(\Lambda)\cap \Lambda =0$.
                              \end{proof}

                              In the  next corollary, we completely characterize all the partial actions of algebras on indecomposable algebras whose enveloping algebras are indecomposable. 

\begin{corollary} Let $\Lambda$ be a unital indecomposable algebra and $\alpha$  a partial action of a group $G$ on $\Lambda$ with enveloping action $(S, \beta)$. Then $S$ is an indecomposable algebra if and only if the partial action $\alpha$ is a global action.\end{corollary} 

\begin{proof} 

Suppose that $S$ is indecomposable.  Then by Proposition \ref{indec} there exists a subgroup $H$ of $G$ such that the partial action $\alpha$ is the extension by zero of the action of $H$  on $\Lambda$. Thus by  Theorem \ref{idemc2}  the algebra $\amalg_{g\in T}\Lambda_g$, where $\Lambda_g=\Lambda$  and $T=\{g_i:i\in I\}$ a transversal of $G/H$,  is  the enveloping action of $(\Lambda, \alpha)$.  Hence, by (\cite{DE}, Theorem 4.5) we have that  $\amalg_{g\in T} \Lambda_g \simeq S$. 
 So,  $G/H$ is trivial and we have that the partial action $\alpha$ is a global action.

The other implication is clear. \end{proof}

Now, we give an example to illustrate our Theorem \ref{idemc2}.

\begin{example} Let $\Lambda$ be the $K$-algebra, where $K$ is a field,  given by the so called the two way quiver. 

\begin{center} $\begin{array}{cccc}
 & \stackrel{\alpha}{\rightarrow} & \cr
\circ 1 & & \circ 2\cr
   & \stackrel{\beta}{\leftarrow} & \cr
\end{array}$. \end{center}

 We define an action of $\mathbb{Z}_2=\{1,\sigma\}$ on $\Lambda$ as follows:

$\left\{\begin{array}{c}
 1\rightarrow 2 \\ 
 2\rightarrow 1
 \end{array} \right.$ and $\left\{\begin{array}{c}
 \alpha \rightarrow \beta\\
 \beta \rightarrow \alpha\end{array}\right. $.

 Note that  $\mathbb{Z}_2$ is isomorphic to  a subgroup of order of 2 of $S_3$ and we can consider in this case  the subgroup $L=\{(1), (12)\}$ of order 2 of $S_3$. We have the left cosets  $\{(1), (1,2)\}$, $\{(2,3), (1,3,2)\}$ and $\{(1,3), (1,2,3)\}$  where the set of representatives of left cosets modulo $L$ is $\{\bar{1}, \bar{(2,3)}, \bar{(1,3)}\}$. Thus, we have the following table where we compute the functions $h$ and $j$ defined as before:  
 
$ \begin{tabular}{|c|c|c|}
 \hline 
   & j & h \\ 
 \hline 
 (1,1) & 1 & 1 \\ 
 \hline 
 (1,(23)) & (23) & 1 \\ 
 \hline 
 (1,(13)) & (13) & 1 \\ 
 \hline 
 ((12),1) & (13) & (12) \\ 
 \hline 
 ((12),(23)) & (13) & (12) \\ 
 \hline 
 ((23),1) & (23) & (23) \\ 
 \hline 
 ((23),(23)) & 1 & 1 \\ 
 \hline 
 ((23),(13)) & (13) & (12) \\ 
 \hline 
 ((123),1) & (13) & (12) \\ 
 \hline 
 ((123),(23)) & 1 & (12) \\ 
 \hline 
 ((123),(13)) & (23) & 1 \\ 
 \hline 
 ((13),1) & (13) & 1 \\ 
 \hline 
 ((13),(23)) & (23) & (12) \\ 
 \hline 
 ((13),(13)) & 1 & 1 \\ 
 \hline 
 ((132),1) & (23) & (12) \\ 
 \hline 
 ((132),(23)) & (13) & 1 \\ 
 \hline 
 ((132),(13)) & 1 & (12) \\ 
 \hline 
 \end{tabular}$.
\vspace{.4cm}
 
  By Theorem \ref{idemc2} we have that the enveloping action is  $(T,\beta)$, where  $T=\Lambda_{1}\times \Lambda_{(13)}\times \Lambda_{23}=\Lambda \times \Lambda \times \Lambda$ and the  global action $\beta$ of $S_3$ is defined  as follows:
\begin{itemize} 
\item $\beta_{1}=id_{T}$

\item $\beta_{(12)}(x,y,z)=((12)x,(12)z,(12)y)$

\item $\beta_{(13)}(x,y,z)=(y,x,(12)z)$

\item $\beta_{(123)}(x,y,z)=((12)z,(12)x,y)$

\item $\beta_{(23)}(x,y,z)=(z,(12)y,x)$

\item $\beta_{(132)}(x,y,z)=((12)y,z,(12)x)$
\end{itemize} 
  \end{example}

Next, we state  a proposition whose proof we give for the sake of completeness.

\begin{proposition}

Let $\Lambda_1$,...,$\Lambda_k$ be unital  indecomposable algebras  such that $\Lambda_i$ is not isomorphic to $\Lambda_j$ for $i\neq j$,   $\Lambda =\Lambda_1^{n_1} \times ... \times\Lambda_k^{n_k}$ and   $\alpha$  a partial action of a group $G$ on $\Lambda$ with an enveloping action. Then  each component $\Lambda_i^{n_i}$ is $\alpha$-invariant and  the restriction of $\alpha$ to the component $\Lambda_i^{n_i}$ defines a partial action on it and the original partial action is a product of these partial actions. Conversely given a set of partial actions $(\Lambda_i^{n_i}, \alpha_i)$ we can define the  partial action  $\alpha$ of $G$ on $\Lambda$, as the product of the given partial actions. 
\end{proposition}
\begin{proof}
  Note that given a central idempotent $e$ of  $\Lambda$, and a decomposition  $e=e_1 + \cdots + e_k$  with $e_i \in \Lambda_i^{n_i}$  and it follows that each  $e_i$ is a central 
idempotent of $\Lambda_i$. All central idempotents of $\Lambda$ are of this form.   We easily have for  each ideal $I$ of $\Lambda$, that  is generated by a central idempotent, is of the form   $I_1\times...\times  I_ k$ where $I_t = \Lambda e_t= \Lambda^{n_t}_te_t$, with $e_t$ a central idempotent in $\Lambda^{n_t}_t$, for each $t\in \{1,...,k\}$.

Now, let $I_i$ be an ideal of $\Lambda_i^{n_i}$ and $I_j$ an ideal of  $\Lambda_{j}^{n_j}$   generated by  central idempotents with $i\neq j$,  then there is no isomorphism between $I_i$ and $I_j$ because of $\Lambda_i$ is not isomorphic to $\Lambda_j$. Thus, $\alpha_g(\Lambda_i^{n_i}\cap S_{g^{-1}})\subseteq \Lambda_i^{n_i}\cap S_g$, for each $g\in G$ and it follows that $\Lambda_i^{n_i}$ is $\alpha$-invariant.  Hence, for each $i\in \{1,..,k\}$,  we have a partial action $\alpha_i$ of $G$ on $\Lambda_i^{n_i}$ and we have that the partial action $\alpha$ is of the form $(\amalg_{i=1}^k \Lambda_i^{n_i}, \amalg_{i=1}^{k} \alpha_i)$.

For the converse, we observe that if an algebra $A$ is of the form $A=B\times C$,  then the product of a partial action $\alpha_1$   of $G$ on $B$  and a partial action $\alpha_2$ of $G$ on $C$ is a partial action of $G$ on $A$.

\end{proof}   

The former proposition tells us that in order to describe all the partial actions on an algebra $\Lambda$, we only need to describe the partial actions on algebras of the form 
$\Lambda^n$, where $\Lambda$ is an unital indecomposable algebra.

  The following proposition, goes in this direction and it is probably well known. We give a proof here, for the sake of completeness.

\begin{proposition}

Let $B=\Lambda^n$ where $\Lambda$ is an unital idecomposable algebra. Then
$Aut(B)\simeq S_n\times Aut(\Lambda)^n$, where $S_n$ is the group of permutations on $n$ elements, $Aut(\Lambda)$ is the group of automorphisms of $\Lambda$ and  $Aut(B)$ the group of automorphims of $B$. 
\end{proposition}

\begin{proof}
Define the map $\psi:S_n\times Aut(\Lambda)^n \to Aut(B)$ by 
$\psi(\eta,f_1,\dots f_n)(\sum \lambda_i e_i) = \sum f_i(\lambda_i)e_{\eta(i)}$. We easily have that $\psi$  is a group isomorphism.
\end{proof}

The next theorem shows that  if a partial action $\alpha$ of a  group $G$ on  $\Lambda^n$, with $\Lambda$ an unital  indecomposable algebra has an enveloping action  $(L, \beta)$ then $L=\oplus_{i\in I}\Lambda_i$, where $\Lambda_i=\Lambda$ with $\sharp I$ is given in the proof.   

\begin{theorem}
Let  $G$ be any group, $\Lambda$  an unital indecomposable algebra and  $(\Lambda^n, \alpha)$  a partial action with an enveloping action $(S,\psi)$.
 Then  the set of the primitive  central idempotents $Y=\{e_1,\dots, e_n\}$ of $\Lambda^n$ is $\alpha$-invariant.  Moreover if $X$ is the enveloping set of the induced partial action  on $Y$, then  $S\simeq \amalg_{i\in I} \Lambda_i$, where $\Lambda_i=\Lambda$, for all $i\in I$ and $I$ is an  index set  such that $\sharp I=\sharp X$ and  in this case we have that $\sharp I\leq  n|G|$, because of  $\sharp X\leq  n|G|$.
 \end{theorem}

\begin{proof}  First, we show that $Y$ is $\alpha$-invariant. In fact, let $g\in G$ and $A=Y\cap D_{g^ {-1}}$. If $D_{g^ {-1}}$ is the zero ideal, then $\alpha_g(Y\cap D_{g^ {-1}})=Y\cap D_g$.  Suppose that $D_{g^ {-1}}\neq 0$ and in this case $Y\cap D_{g^ {-1}}\neq \emptyset$, since $D_{g^ {-1}}$ is a direct sum of ideals of $\Lambda^n$.  Thus,  $1_{g^{-1}}=e_{i_1}+...+e_{i_s}$ and we clearly have  that $e_{i_1},...,e_{i_s}\in D_{g^ {-1}}$.  Hence, $1_g=\alpha_g(e_{i_1})+...+\alpha_g(e_{i_s})$. By the fact that $1_g=e_{j_1}+...+e_{j_k}$ and the elements of the set $\{\alpha_g(e_{i_1}),...,\alpha_g(e_{i_s}\}$ are primirive central idempotents of $D_g$ we have that $k=s$ and $\alpha_g(e_{i_p})=e_{j_k}$. Consequently, $\alpha_g(e_{i_p})$ is a primitive central idempotent of $\Lambda^ n$, for all $p\in \{1,..,n\}$. So, $\alpha_g(Y\cap D_{g^ {-1}})\subseteq Y\cap D_g$ and it follows that  $Y$ is $\alpha$-invariant.      

We define the partial action  $\gamma$ of $G$ on $\{e_1,\dots, e_n\}$  by the restriction of the partial action $\alpha$ to the set  $\{e_1,\dots, e_n\}$.

Next, we consider   $(X,\theta)$ be the enveloping action of the partial action $\gamma$ on $Y=\{e_1,\dots, e_n\}$, where  $X$ is the $\Psi$-orbit of $Y$ and $\theta$ is the restriction  of  $\beta$ to $X$   We consider $L=\amalg_{i\in I} \Lambda_i$, where $\Lambda_i=\Lambda$, for all $i\in I$ and $I$ is an  index set such that $\sharp I=\sharp X$. It is well know that there exists an injective homomorphism  $i: \Lambda^{n}\rightarrow \amalg_{i\in I} \Lambda_i$.
  
We define a global action $\beta$ of   $G$ on $L$  by $\beta_g(\sum \lambda_i e_i)=\sum\psi_g(\lambda_i)\theta_g(e_i)$, where $\Psi_g(\lambda_i)=\Psi_g(\lambda_i e_i)$.  Note that $\beta|_{\{e_1,\dots, e_n\}}$ gives the partial action $\gamma$. Then the restriction of the action  $\beta$ to $\Lambda^n$ is the  partial action $\alpha$ and it follows that $(L,\beta)$ is the  enveloping action of the $(\Lambda^n, \alpha)$. So, by (\cite{DE}, Theorem 4.5) we have that   $S\simeq \amalg_{i\in I} \Lambda_i$. 

Moreover, all the partial actions which restricts to $\gamma$ are of this form. \end{proof}

As a consequence of the last theorem we have the following result.

   \begin{corollary}
	Let $\alpha$ and $\alpha'$ be globalizable partial actions of  $G$ on $\Lambda^n$ where $\Lambda$ is an unital indecomposable algebra,  $ ( \Lambda^m, \beta)$ and
	$( \Lambda^{m'}, \beta')$  the respective enveloping actions. Assume that both partial actions induce the same partial action on the set of primitive central idempotents  $\{e_1,...,e_n\}$. Then $m=m'$.
	\end{corollary}

In the literature there are a lot of examples  of partial actions of groups on algebras of type $\Lambda=\amalg_{i=1}^nKe_i$, where $\{e_i:1\leq i\leq n\}$ is the set of primitive central idempotents and $K$ is a field, see \cite{CCF}, \cite{MWM} and the references therein.  We study during the rest of this section   the enveloping actions of the partial  actions of  any group $G$ on $\Lambda=\amalg_{i=1}^nKe_i$, where $K$ is a field and all partial automorphisms associated to the partial action are $K$-linear.

We note that for any ideal $I$ of the algebra $\Lambda=\amalg_{i=1}^nKe_i$ there is a subset $T$ of $\{e_i: 1\leq i\leq n\}$ such that 
 $I=\amalg_{e\in T}Ke$.  Let $X$ be the set of ideals of $\Lambda$. Then there is a bijection  $\psi:X\rightarrow \mathcal{P}(\{e_i:1\leq i\leq n\})$ defined by $\psi(0)=\emptyset$ and for $I\neq 0$, $\psi(I)=\{e_i:e_iI\neq 0\}$. 
 
 The following lemma is well known and we include it here for the sake of completeness.

\begin{lemma}\label{bijection} Let $\Lambda$ as above and $I_1$, $I_2$ ideals of $\Lambda$. Then $I_1\cong I_2$ if and only if $\sharp \psi( I_1)=\sharp \psi(I_2)$.

%where $\sharp \ \psiI_j$, for $j=1,2$, means the number of the central primitive idempotents $e$ such that $Ie_j=I$.

\end{lemma}

\begin{proof} Suppose that $\sharp \psi(I_1)=\sharp \psi(I_2)$. Then we easily have that  there exists  a bijection $\theta:\psi(I_1)\rightarrow \psi(I_2)$.

% where $i(I_j)$, $j\in \{1,2\}$, is the set of primitive central idempotents $e$ of $\Lambda$ such that $I_je=I_j$. 

Thus,  the application  $\Psi:I_1\rightarrow I_2$ defined by \begin{center}$\Psi(\sum_{x\in i(I_1)}\lambda_x e_x)=\sum_{x\in i(I_1)}\lambda_{\theta(x)}e_{\theta(x)}$\end{center} is an isomorphim between $I_1$ and $I_2$.

The converse is trivial.

 \end{proof} 
 
 The next proposition was proved in (\cite{BDG}, Proposition 3.6) and we put it here in our context for the sake of completeness. Moreover, it is  convenient to point out  that   for a  partial action of $G$ on $\Lambda$ is  equivalent to give a partial action on the set $\{e_i:1\leq i\leq n\}$. 

\begin{proposition} Let $\Lambda=\amalg_{i=1}^n Ke_i$ and  $G$  a group. Then,   all  the partial actions of $G$ on $\Lambda$ can be restricted to the set $\{e_i\}_{i=1}^n$. Moreover, we obtain a bijection between the sets of  the partial actions of $G$ on $\Lambda$ and  partial actions of $G$ on $\{e_i\}_{i=1}^n$.\end{proposition}

%\begin{proof}  Let $\alpha=\{\alpha_g:I_{g^{-1}}\rightarrow I_{g}: g\in G\}$ be a partial action of $G$ on $\Lambda$. We define  $\overline{\alpha}=\{\overline{\alpha}_g:i(I_{g^{-1}})\rightarrow i(I_g)\}$. It is not difficult to show  that  $\overline{\alpha}$ is a partial action of $G$ on $\{e_i\}_{i=1}^{n}$. Now,  given any partial action of $G$ on  $\{e_i\}_{i=1}^{n}$ and using Lemma 3.13 we obtain a partial action of $G$ on $\Lambda$.  The result follows. \end{proof}

In (\cite{EDJ2}, Proposition 8.4)  the authors proved that the enveloping actions of twisted partial actions of groups  on  algebras decomposed by blocks are  algebras decomposed by blocks, if these  algebras are unital, but in our final main result, we do not need the assumption that the enveloping action has unit  to show that it is an  algebra decomposed by  blocks.  

\begin{theorem} Let $\Lambda=\amalg_{i=1}^nKe_i$  be as above and $\alpha$ a partial action of $G$ on $\Lambda$. Then the enveloping action is $B=\sum_{e\in X}Ke$, where $X$ is the enveloping action of the induced partial action $\overline{\alpha}$ of $\alpha$ on $\{e_i :1\leq i\leq n\}$.\end{theorem}

\begin{proof} We define the global action $\beta^1$ of $G$ on $B$ by \begin{center} $\beta^{1}_{g}(\sum_{e\in X}\lambda_e e)=\sum_{e\in X}\lambda_e \beta_g(e)$,\end{center}  where $\beta$ is the global action of $G$ on $X$.  We leave to the reader to show that  $(B,\beta^{1})$ is the enveloping action of $(\Lambda,\alpha)$. \end{proof}

{\bf Thanks:} We are grateful to professor M. Dokuchaev that warned us  about the article \cite{EDJ2}  where the authors studied  globalization of  twisted partial actions of groups on algebras decomposed by blocks.

\end{document}